\documentclass[11pt,reqno]{amsart}
\usepackage{graphicx}
\usepackage{amssymb,amsmath}
\usepackage{amsthm}
\usepackage{color,graphicx}
\usepackage{hyperref}
\usepackage{color}
\usepackage{epstopdf}

\newcommand{\be}{\begin{equation}}

\newcommand{\ee}{\end{equation}}

\newcommand{\R}{{\mathbb R}}

\newcommand{\I}{{\mathcal{I}}}

\newcommand{\RR}{{\mathcal{R}}}
\newcommand{\ve}{{\varepsilon}}

\newcommand{\ra}{{\rangle}}
\newcommand{\la}{{\langle}}

\numberwithin{equation}{section}
\numberwithin{figure}{section}

\newtheorem{theorem}{Theorem}[section]
\newtheorem{proposition}[theorem]{Proposition}
\newtheorem{remark}[theorem]{Remark}
\newtheorem{lemma}[theorem]{Lemma}
\newtheorem{corollary}[theorem]{Corollary}
\newtheorem{definition}[theorem]{Definition}

\begin{document}
\vglue-1cm \hskip1cm
\title[Orbital Stability of Periodic Waves]{SUFFICIENT CONDITIONS FOR ORBITAL STABILITY OF PERIODIC TRAVELING WAVES}

\begin{center}



\maketitle

{\bf Giovana Alves }

{Departamento de Matem\'atica - Universidade Estadual de Maring\'a\\
	Avenida Colombo, 5790, CEP 87020-900, Maring\'a, PR, Brazil.}\\
{a\underline{ }\underline{ }giovanaalves@yahoo.com.br}

\vspace{3mm}

{\bf F\'abio Natali}

{Departamento de Matem\'atica - Universidade Estadual de Maring\'a\\
Avenida Colombo, 5790, CEP 87020-900, Maring\'a, PR, Brazil.}\\
{ fmnatali@hotmail.com}

\vspace{3mm}

{\bf Ademir Pastor}

{IMECC-UNICAMP\\
Rua S\'ergio Buarque de Holanda, 651, CEP 13083-859, Campinas, SP,
Brazil.}\\
{ apastor@ime.unicamp.br}

\end{center}

\begin{abstract}
The present paper deals with sufficient conditions for orbital stability of periodic waves of a  general class of evolution equations supporting  nonlinear dispersive waves.  Our method can be seen as an extension to spatially periodic waves of the theory of solitary waves recently developed in \cite{st}. Firstly, our main result do not depend on the parametrization of the periodic wave itself. Secondly, motived by the well known orbital stability criterion for solitary waves, we show that the same criterion holds for periodic waves. In addition, we show that the positiveness of the principal entries of the Hessian matrix related to the ``energy surface function'' are also sufficient to obtain the  stability. Consequently, we can establish the orbital stability of periodic waves for several nonlinear dispersive models. We believe our method can be applied in a wide class of evolution equations; in particular it can be extended to regularized dispersive wave equations.
\end{abstract}

\section{Introduction}
Over the last few years, the study of stability of periodic traveling waves associated
with nonlinear dispersive equations has had a considerable increase. A rich variety of new mathematical problems as well as  physical applications  have emerged. This subject is often studied in connection to the natural symmetries associated to the model (translations  and/or rotations). The perturbations can be taken with respect to several classes, e.g., (i) the class of periodic functions with the same or a multiple of the minimal period as the underlying wave; (ii) the class of localized functions. In the case of shallow-water wave
models (or long internal waves in a density-stratified ocean, ion-acoustic waves in a plasma or acoustic waves on a crystal lattice), it is well known that a formal
stability theory of periodic traveling waves has started with the pioneering work of Benjamin \cite{benjamin}, who considered  the periodic steady solutions (the so called cnoidal
waves) for the Korteweg-de Vries equation (KdV henceforth),
\begin{equation}\label{kdv2}
u_t+uu_x+u_{xxx}=0.
\end{equation}

In this paper, we present sufficient conditions for the orbital stability of periodic traveling-wave solutions related to the following generalization of $(\ref{kdv2})$,
\be\label{rDE}
u_t+(f(u))_x-(\mathcal{M}u)_x=0,
\ee
where $f:\mathbb{R}\rightarrow\mathbb{R}$ is a smooth function and $u:\R\times\R\to\R$ is spatially periodic with period $L>0$. Here $\mathcal{M}$ is a differential or pseudo differential operator which may be defined as a Fourier multiplier by 
\be\label{symbol123}
\widehat{\mathcal{M}g}(\kappa)=\theta(\kappa)\widehat{g}(\kappa), \quad \kappa\in\mathbb{Z},
\ee
where $\theta$ is assumed to be an even and continuous function on $\R$ satisfying
\begin{equation}\label{A1A2}
\upsilon_1|\kappa|^{m}\leq\theta(\kappa)\leq \upsilon_2|\kappa|^{m}, \quad m>0,
\end{equation}
for  $|\kappa|\geq \kappa_0$ and for some $\upsilon_i>0$, $i=1,2$.

 Formally, equation $(\ref{rDE})$ admits the conserved quantities (to this end, a convenient local well-posedness result is welcome to assure them),
\begin{equation}\label{Eu}
P(u)=\frac{1}{2}\int_{0}^{L}\Big(u\mathcal{M}u-W(u)\Big)dx,
\end{equation}
\begin{equation}\label{Fu}
F(u)=\frac{1}{2}\int_{0}^{L}u^2dx,
\end{equation}
and 
\begin{equation}\label{Mu}
M(u)=\int_{0}^{L}udx.
\end{equation}
Here, function $W$ denotes the primitive of $f$, that is, $W'=f$.

Periodic traveling waves for (\ref{rDE}) are solutions of the form $u(x,t)=\phi(x-\omega t)$, where $\omega\in \R$ and $\phi:\R\to\R$ is a smooth periodic function. Substituting this form into (\ref{rDE}), we obtain
\begin{equation}\label{ode-wave}
\mathcal{M}\phi+\omega\phi-f(\phi)+A=0,
\end{equation}
where $A$ is a constant of integration.

The  quantities in $(\ref{Eu})$-$(\ref{Mu})$ allow us to define two new conserved quantities. The first one is the constrained energy 
\begin{equation}\label{lyafun}
G(u)=P(u)+\omega F(u)+AM(u),
\end{equation}
and the second one is the auxiliary quantity
\begin{equation}\label{modquant}
Q(u)=\mu M(u)+ \nu F(u), 
\end{equation}
where $\mu, \nu$ are real constants and at least one of them are nonzero. For the moment, the parameters $\mu$ and $\nu$ do not play any role; however, as we will see in our stability criteria, they can be appropriately chosen in order to obtain the positiveness of the linearized operator.

It is well known that the linearization of $(\ref{rDE})$  around a periodic wave $\phi$ gives rise to the  linear operator
\begin{equation}\label{operator}
\mathcal{L}=\mathcal{M}+\omega-f'(\phi),
\end{equation}
which shall be considered on $L_{per}^2([0,L])$. The relation in \eqref{A1A2} allows to say that $\mathcal{L}$ is a well-defined operator with domain $D(\mathcal{L})=H_{per}^m([0,L])$.
Under our assumption ($H_1$) below, it is not difficult to see the connection between $\mathcal{L}$ and $G$. Indeed, the quadratic form $v\mapsto\langle G''(\phi)v,v\rangle$ is closed, densely defined, and bounded from below on $H_{per}^{\frac{m}{2}}([0,L])$. Consequently,  $\mathcal{L}$ is the unique  self-adjoint linear operator such that (see e.g. \cite[Chapter VI]{kato})
\begin{equation}\label{dual1}
\langle G''(\phi)v,z\rangle=(\mathcal{L}v,z),\ \ \ \ \ \ v\in C_{per}^{\infty}([0,L]),\ z\in H_{per}^{\frac{m}{2}}([0,L]),
\end{equation}
where $G''$ represents the second order Fr\'echet derivative of $G$, $\langle\cdot,\cdot\rangle$ is the duality in $H_{per}^{-\frac{m}{2}}([0,L])$ and $(\cdot,\cdot)$ indicates the inner product in $L_{per}^2([0,L])$. In particular, $G''(\phi)v=\mathcal{I}\mathcal{L}v$, for all $v\in C_{per}^{\infty}([0,L])$, where $\mathcal{I}: H_{per}^{\frac{m}{2}}([0,L])\rightarrow H_{per}^{-\frac{m}{2}}([0,L])$ is the natural injection of $H_{per}^{\frac{m}{2}}([0,L])$ into $H_{per}^{-\frac{m}{2}}([0,L])$ with respect to inner product in $L_{per}^2([0,L])$, that is,
\begin{equation}\label{dual2}
\langle \mathcal{I}u,v\rangle=(u,v),\ \ \ \ \ \ u,\ v\in H_{per}^{\frac{m}{2}}([0,L]).
\end{equation}

Next, we present a brief outline of our work. The main assumptions we assume throughout the paper are the following:

\begin{enumerate}
\item[$(H_0)$] Suppose that $\phi\in C_{per}^{\infty}([0,L])$ is an $L$-periodic traveling wave solution of (\ref{ode-wave}) with $L>0$. The operator $\mathcal{L}$ has only one negative eigenvalue which is simple and zero is a simple eigenvalue whose eigenfunction is $\phi'$.
\end{enumerate}
\begin{enumerate}
\item[$(H_1)$] There exist constants $c_1, c_2>0$ such that
$$(\mathcal{L}v,v)\geq c_1 ||v||_{\frac{m}{2}}^2-c_2||v||^2,$$ for all $v\in C_{per}^{\infty}([0,L])$,
	\end{enumerate}

\begin{enumerate}
	\item[$(H_2)$] There exists $c_3>0$ such that
	$$(\mathcal{L}v,v)\geq c_3 ||v||^2,$$ for all $v\in H_{per}^{m}([0,L])$ satisfying $(v,\phi')=(v,Q'(\phi))=0$.
	\end{enumerate}

Our main goal in the present paper is to show that, under hypothesis $(H_0)-(H_2)$, the periodic traveling wave $\phi$ is orbitally stable in the energy space $H_{per}^{\frac{m}{2}}([0,L])$ by the periodic flow of  $(\ref{rDE})$. From \eqref{ode-wave} and \eqref{lyafun}, $\phi$ is a critical point of the constrained energy $G$, that is $G'(\phi)=0$. Thus, it is expected that the Hessian $G''(\phi)$ determines the stability of $\phi$. From \eqref{dual1} we see that assumption ($H_2$) implies that $\phi$ is a minimum of $G$ when restricted to a suitable codimension two manifold.  Therefore, our results are based on the construction of a convenient Lyapunov functional. The arguments used here follow the approach in \cite{st}, where the author established the orbital stability of standing waves for abstract Hamiltonian systems of the form
\begin{equation}\label{stha}
u_t(t)=JE'(u(t)),
\end{equation}
posed on a Hilbert space $X$, where $J$ is an invertible skew-symmetric  bounded operator on $X$ and $E$ is the  associated energy. In particular, it is assumed  that \eqref{stha} is invariant under the action of a one-dimensional group. This enabled the author to prove the orbital stability of standing waves for a large class of nonlinear Schr\"odinger-type equation with a potential. It should be noted, however, that the general theory in \cite{st} cannot be directly  applied to our case. In fact, even though $(\ref{rDE})$ can be written as an Hamiltonian system with $J=\partial_x$, such an operator is clearly not invertible on $H^{\frac{m}{2}}_{per}([0,L])$.  Hence, in the present paper,  we modify the approach in \cite{st} in order to consider the case when $J$ is not invertible. As we will see below, this allows us to obtain new results concerning the stability of periodic waves.

The existence and orbital stability of  periodic waves for equations like \eqref{rDE} has gained much more attention after the work in \cite{ABS}. In that paper, the authors considered the KdV equation and proved the existence and orbital stability of explicit periodic waves with the zero mean property, say, $\omega\mapsto\phi_\omega$, with $\omega$ belonging to an unbounded interval. Fundamentally, the authors have used the fact that the Hessian matrix 
$d''(\omega):=\frac{d}{d\omega}F(\phi_\omega)$ is positive, where $d$ is the energy curve function defined by $d(\omega)=P(\phi_\omega)+\omega F(\phi_\omega)$.

 In \cite{johnson}, the author established sufficient conditions for the orbital stability of periodic waves to the generalized KdV equation 
\begin{equation}\label{gBBM}
u_t+u^pu_x+u_{xxx}=0,
\end{equation}
where $p\geq1$ is an integer. He has constructed smooth periodic waves $\phi(\cdot,A,B,\omega)$ whose period depends smoothly on the triple $(A,B,\omega)\in\widetilde{\mathcal{O}}$, where $\widetilde{\mathcal{O}}\subset\mathbb{R}^3$ is an open set. Here $B$ is an integration constant which appears in the quadrature form associated with the second order differential equation in $(\ref{ode-wave})$ (with $\mathcal{M}=-\partial_x^2$ and $f(v)=\frac{v^{p+1}}{p+1}$) and can be interpreted as the associated energy. So, by assuming that the sign of the Jacobian determinants $L_{B}$, $\{L,M\}_{A,B}$ and $\{L,M,F\}_{A,B,\omega}$ (in the notation of that paper) are positive at the point $(A_0,B_0,\omega_0)\in\widetilde{\mathcal{O}}$, one has the orbital stability of $\phi(\cdot,A_0,B_0,\omega_0)$ (see also \cite{BJK1}). In particular, the positiveness of such determinants were checked in the solitary wave limit (whether $1\leq p<4$)  and for periodic waves near the equilibrium solution.

\begin{remark}
It should be noted the positiveness of $L_B$ implies that $\mathcal{L}$ satisfies the assumptions in ($H_0$), while the positiveness of $\{L,M\}_{A,B}$ and $\{L,M,F\}_{A,B,\omega}$ imply that $\mathcal{L}$ is a positive operator on a convenient manifold (see Lemmas 4.2 and 4.4 in \cite{johnson}).
\end{remark}

If a suitable parametrization of the solutions of \eqref{ode-wave} is available, the general strategy for proving the orbital stability, roughly speaking, consists into two steps: first one proves the stability with respect to perturbations in a convenient manifold, frequently supporting restrictions on the conserved quantities, and, second, one extends the class of perturbations to the whole space by using, for instance, a triangle inequality argument; e.g. \cite{angulo1}, \cite{natali}, \cite{an},  \cite{natali1}, \cite{BJ},  \cite{BJK1}, \cite{gss1}, \cite{gss2}, \cite{hur}, \cite{johnson}, \cite{johnson1} to cite but a few. To the latter, is often necessary to use that the Hessian matrix of the energy surface function (or some quantities involving Jacobian determinants as in \cite{johnson}) is non-singular. Generally speaking, this verification turns out to be a hard task.
For one hand, the method we employ here has the advantage that it does not require this kind of information and, in particular, a smooth parametrization of the periodic solutions of \eqref{ode-wave} is not needed. This advantage is generated by the construction of a special Lyapunov functional (see Section \ref{sec3}).  On the other hand, at a first glance, our method does not apply when the spectral properties for the linearized operator in ($H_0$) are not met or when one consider periodic perturbation with a multiple of the minimal period of the underlying wave or localized perturbations (this will subject for further investigation). It is to be highlighted that, from this last point of view, the works \cite{BJK1},  \cite{DK1}, \cite{johnson}, \cite{DK} are more general than our.

Our results are also closed connected with spectral stability. In \cite{BJK1}, \cite{BJK2} and \cite{DK}, the authors proved sufficient conditions for the spectral stability associated with the generalized KdV equation (the results also apply for a large class of  models as those in $(\ref{rDE})$). The criteria are bases on the so called Krein-Hamiltonian stability index
\begin{equation}\label{krein}
K_{{\rm Ham}}=k_i^{-}+k_c+k_r,
\end{equation}
where $k_r$ stands for the number of real eigenvalues of $\partial_x\mathcal{L}$ in the open right-half plane, $k_c$ is the number of complex eigenvalues in the open right-half plane, and $k_i^{-}$ is the total negative Krein signature (see \cite{BJK1} for the precise definition). In particular, if $K_{\rm Ham}=0$ then one has the spectral stability\footnote{If $K_{\rm Ham}=0$ then there are no purely imaginary eigenvalues of negative Krein signature. This is enough to show that spectral stability implies orbital stability; see e.g., \cite{BJK1} and \cite{DK}.}.
The results in the above mentioned papers mainly concerns in establishing simple expressions for the index $K_{\rm Ham}$. In particular, it is assumed that $(\mathcal{L}^{-1}1,1)\neq0$ (note that the assumption $\{L,M\}_{B,A}$ in \cite{BJK1} also reduces to $(\mathcal{L}^{-1}1,1)\neq0$; see Remark 3.5 in that paper). The arguments in \cite{DK} require, besides $(\mathcal{L}^{-1}1,1)\neq0$, that
\begin{equation}\label{Dmatrix}
\mathcal{D}=\frac{1}{(\mathcal{L}^{-1}1,1)}\left|\begin{array}{llll}
(\mathcal{L}^{-1}\phi,\phi) & & (\mathcal{L}^{-1}\phi,1)\\\\
(\mathcal{L}^{-1}\phi,1) & & (\mathcal{L}^{-1}1,1)\end{array}\right|\neq0.
\end{equation}

As we will see below (Theorem \ref{teostab}), in our approach  it is not necessary to assume  $(\mathcal{L}^{-1}1,1)\neq0$. In fact, if there exists a smooth surface of periodic waves for \eqref{ode-wave}, say, $(\omega,A)\in\mathcal{O}\subset\mathbb{R}^2\mapsto\phi=\phi_{(\omega,A)}\in C_{per}^{\infty}([0,L])$, with fixed period $L>0$, we will see that the orbital stability can be determined whether one of the following assumptions holds: $(\mathcal{L}^{-1}1,1)=-M_{A}(\phi)<0$, $(\mathcal{L}^{-1}\phi,\phi)=-F_{\omega}(\phi)<0$ or $M_{\omega}(\phi)^2-F_{\omega}(\phi)M_A(\phi)>0$. In the second case, we recover the criterion used for the case of solitary waves as determined in \cite{bo} and \cite{gss1}. 
%%As a crucial case study, we show that the periodic \textit{dnoidal} waves for the Kawahara equation
%%\begin{equation}\label{kawa}
%%u_t+uu_x+u_{xxx}-u_{xxxxx}=0,
%%\end{equation}
%%are orbitally stable but $(\mathcal{L}^{-1}1,1)=0$ (as we have already mentioned, this last information prevents us in using the arguments in \cite{BJK2} and \cite{DK}).

\indent Our paper is organized as follows. In Section \ref{PL}, we present results concerning the positivity of the linearized operator $\mathcal{L}$. Section \ref{sec3} is devoted to show the orbital stability of periodic waves related to the general model in $(\ref{rDE})$. Applications will be presented in Section \ref{sec4}. In Section \ref{sec5}, we see that the arguments  in Sections \ref{PL} and \ref{sec3} can be used to show a similar criterion for the regularized version of $(\ref{rDE})$.
\bigskip

\noindent \textbf{Notation.} In what follows, we denote by $||\cdot||_{s}$ and $(\cdot,\cdot)_s$ as the norm and the inner product in $H_{per}^s([0,L])$, $s\geq0$. For short, we set $||\cdot||_0:=||\cdot||$ and $(\cdot,\cdot)_0:=(\cdot,\cdot)$.

\section{Positivity of the operator $\mathcal{L}$}\label{PL}

We start this section with some technical and useful results regarding the operator $\mathcal{L}$ and its dual representation $G''(\phi)$. First, in order to simplify the notation, let us define 
$$\mathcal{X}=\{u\in H_{per}^{\frac{m}{2}}([0,L]);\ (v,\phi')=(v,Q'(\phi))=0\}.$$ 

\begin{lemma}\label{lema1}
Suppose that assumptions $(H_1)$ and $(H_2)$ hold. There exists $c_4>0$ such that
\begin{equation}\label{est1}
\langle G''(\phi)v,v\rangle\geq c_4 ||v||_{\frac{m}{2}}^2,
\end{equation}
for all $v\in\mathcal{X}$.
\end{lemma}
\begin{proof}
By density, it suffices to assume that $v$ belongs to $C_{per}^{\infty}([0,L])$. Assumptions $(H_1)$ and $(H_2)$ give us
\begin{equation}\label{est2}
\left(1+\frac{c_2}{c_3}\right)(\mathcal{L}v,v)\geq c_1 ||v||_{\frac{m}{2}}^2.
\end{equation}
The conclusion then follows from $(\ref{dual1})$.
\end{proof}

\begin{remark}\label{garding}
Assume $v\in C_{per}^{\infty}([0,L])$. If $\omega>0$, one has from $(\ref{symbol123})$, the smoothness of $f$, and $(\ref{A1A2})$ that
$$
(\mathcal{L}v,v)=\displaystyle\int_0^L\Big(v\mathcal{M}v+\omega v^2-f'(\phi)v^2\Big)dx\geq  D_1||v||_{\frac{m}{2}}^2-D_2||v||^2,
$$
where $D_i$, $i=1,2$ are positive constants that do not depend on $v$. Thus, inequality in $(H_1)$ holds in this particular case.
\end{remark}

 Now, let $\mathcal{R}: H_{per}^{\frac{m}{2}}([0,L])\rightarrow H_{per}^{-\frac{m}{2}}([0,L])$ be the Riesz isomorphism with respect to inner product in $H_{per}^{\frac{m}{2}}([0,L])$, that is, 
\begin{equation}\label{dual3}
\la \mathcal{R}u,v\ra =(u,v)_{\frac{m}{2}},\ \ \ \ \ \ \ \ u,\ v\in H_{per}^{\frac{m}{2}}([0,L]).
\end{equation}

 Lemma $\ref{lema1}$ establishes the positivity of $G''(\phi)$ under an orthogonality in $L_{per}^2([0,L])$. The next result shows that the same positivity holds if we assume the orthogonality in $H_{per}^{\frac{m}{2}}([0,L])$.

\begin{lemma}\label{lema2}
Let $\mathcal{I}$ be the operator defined in $(\ref{dual2})$. Let
\begin{equation}\label{orto1}
\begin{array}{lllll}\mathcal{Z}&=&\{\phi',\mathcal{R}^{-1}\mathcal{I}Q'(\phi)\}^{\bot}\\
&=&\{z\in H_{per}^{\frac{m}{2}}([0,L]);\ (z,\phi')_{\frac{m}{2}}=(z,\mathcal{R}^{-1}\mathcal{I}Q'(\phi))_{\frac{m}{2}}=0\}
\end{array}
\end{equation}
Then, there exists $c_5>0$ such that 
\begin{equation}\label{est3}
\langle G''(\phi)z,z\rangle\geq c_5 ||z||_{\frac{m}{2}}^2,
\end{equation}
for all $z\in \mathcal{Z}$.
\end{lemma}
\begin{proof}
Let  $\psi=\frac{\phi'}{||\phi'||}$. Take any $z\in \mathcal{Z}$ and define
$$v:=z-(z,\psi)\psi.$$
Let us show that $v\in \mathcal{X}$. In fact, since $||\psi||=1$ one has $(v,\psi)=(z,\psi)-(z,\psi)(\psi,\psi)=0$. Moreover, the fact that $Q'(\phi)=\mu+\nu\phi$ enable us to deduce
$$(v,Q'(\phi))=(\mathcal{R}^{-1}\mathcal{I}Q'(\phi),v)_{\frac{m}{2}}=-(z,\psi)( \mathcal{R}^{-1}\mathcal{I}Q'(\phi),\psi)_{\frac{m}{2}}=-(z,\psi)(\psi,Q'(\phi))=0.$$
An application of Lemma $\ref{lema1}$ yields the existence of $c_4>0$ such that
\begin{equation}\label{est4}
\la G''(\phi)v,v\ra\geq c_4||v||_{\frac{m}{2}}^2,
\end{equation}
that is,
\begin{equation}\label{est5}
(\mathcal{R}^{-1}G''(\phi)v,v)_{\frac{m}{2}}\geq c_4||v||_{\frac{m}{2}}^2.
\end{equation}
Next, since $\phi$ is a smooth function and $\mathcal{L}\phi'=0$, from $(\ref{dual1})$ one infers that $G''(\phi)\phi'=0$. Therefore, from the definition of $v$ one has 
\begin{equation}\label{Gv}
G''(\phi)v=G''(\phi)z.
\end{equation}
Let $S:H_{per}^{\frac{m}{2}}([0,L])\rightarrow H_{per}^{\frac{m}{2}}([0,L])$ be the self-adjoint operator defined by $S=\mathcal{R}^{-1}G''(\phi)$. It is not difficult to prove that $(Sz,z)_{\frac{m}{2}}=(Sv,v)_{\frac{m}{2}}$. In addition, since $z\in\mathcal{Z}$, we have from Cauchy-Schwarz inequality,
$$||z||_{\frac{m}{2}}^2=(z,v+(z,\psi)\psi)_{\frac{m}{2}}=(z,v)_{\frac{m}{2}}\leq ||z||_{\frac{m}{2}}||v||_{\frac{m}{2}},$$
that is, $||z||_{\frac{m}{2}}\leq ||v||_{\frac{m}{2}}$. Finally, combining the last inequality with $(\ref{est5})$,
\begin{equation}\label{est6}(Sz,z)_{\frac{m}{2}}=(Sv,v)_{\frac{m}{2}}\geq c_4||v||_{\frac{m}{2}}^2\geq c_4||z||_{\frac{m}{2}}^2.\end{equation}
Thus, from $(\ref{dual3})$ and $(\ref{est6})$ we obtain the desired result.
\end{proof}

\begin{lemma}\label{lema3}
There are positive constants $\sigma$ and $c_6$ such that
$$(Sv,v)_{\frac{m}{2}}+2\sigma(\mathcal{R}^{-1}\mathcal{I}Q'(\phi),v)_{\frac{m}{2}}\geq c_6||v||_{\frac{m}{2}}^2,$$
for all $v\in\{\phi'\}^{\bot}=\{u\in H_{per}^{\frac{m}{2}}([0,L]);\ (u,\phi')_{\frac{m}{2}}=0\}$.
\end{lemma}
\begin{proof}
In fact, from $(\ref{dual2})$ and $(\ref{dual3})$, we infer that 
\begin{equation}
(\mathcal{R}^{-1}\mathcal{I}Q'(\phi),\phi')_{\frac{m}{2}}=\la \mathcal{I}Q'(\phi),\phi'\ra=(Q'(\phi),\phi')=0.
\label{orto2}\end{equation}
Let $w=\frac{\mathcal{R}^{-1}\mathcal{I}Q'(\phi)}{||\mathcal{R}^{-1}\mathcal{I}Q'(\phi)||_{\frac{m}{2}}}$. Thus, given any $v\in \{\phi'\}^{\bot}$, we define
$$z=v-\alpha w,$$
where $\alpha=(v,w)_{\frac{m}{2}}$. It is easy to see that $z\in\mathcal{Z}$. Thus, Lemma $\ref{lema2}$ implies
\begin{equation}\label{est7}
\begin{array}{llll}
(Sv,v)_{\frac{m}{2}}&=&\alpha^2(Sw,w)_{\frac{m}{2}}+2\alpha(Sw,z)_{\frac{m}{2}}+(Sz,z)_{\frac{m}{2}}\\\\
&\geq& \alpha^2(Sw,w)_{\frac{m}{2}}+2\alpha(Sw,z)_{\frac{m}{2}}+c_5||z||_{\frac{m}{2}}^2.
\end{array}
\end{equation}
But, from Cauchy-Schwarz and Young's inequalities,
$$2\alpha(Sw,z)_{\frac{m}{2}}\leq \frac{c_5}{2}||z||_{\frac{m}{2}}^2+\frac{2\alpha^2}{c_5}||Sw||_{\frac{m}{2}}^2.$$
Therefore, 
\begin{equation}\label{est8}
(Sv,v)_{\frac{m}{2}}\geq \alpha^2(Sw,w)_{\frac{m}{2}}-\left(\frac{c_5}{2}||z||_{\frac{m}{2}}^2+\frac{2\alpha^2}{c_5}||Sw||_{\frac{m}{2}}^2\right)+c_5||z||_{\frac{m}{2}}^2.
\end{equation}
Let $\beta:=||\mathcal{R}^{-1}\mathcal{I}Q'(\phi)||_{\frac{m}{2}}$. One has, 
$$(\mathcal{R}^{-1}\mathcal{I}Q'(\phi),v)_{\frac{m}{2}}=\alpha\beta.$$
Now, choose $\sigma>0$ large enough such that
\begin{equation}
(Sw,w)_{\frac{m}{2}}-\frac{2}{c_5}||Sw||_{\frac{m}{2}}^2+2\sigma\beta^2\geq \frac{c_5}{2}
\label{est9}\end{equation}
It is clear that $c_5$ does not depend on $v$. Hence,
\begin{equation}\label{est10}
\begin{array}{llll}
(Sv,v)_{\frac{m}{2}}&+&\displaystyle 2\sigma(\mathcal{R}^{-1}\mathcal{I}Q'(\phi),v)_{\frac{m}{2}}\\\\&\geq&\displaystyle\alpha^2(Sw,w)_{\frac{m}{2}}-\left(\frac{c_5}{2}||z||_{\frac{m}{2}}^2+\frac{2\alpha^2}{c_5}||Sw||_{\frac{m}{2}}^2\right)+c_5||z||_{\frac{m}{2}}^2+2\sigma\alpha^2\beta^2\\\\
&=&\displaystyle\alpha^2\left((Sw,w)_{\frac{m}{2}}-\frac{2}{c_5}||Sw||_{\frac{m}{2}}^2+2\sigma\beta^2\right) +\frac{c_5}{2}||z||_{\frac{m}{2}}^2\\\\
&\geq&\displaystyle\frac{c_5}{2}\left(\alpha^2+||z||_{\frac{m}{2}}^2\right)\\\\
&=&\displaystyle\frac{c_5}{2}||v||_{\frac{m}{2}}^2.
\end{array}
\end{equation}
The result is thus proved with $c_6=\frac{c_5}{2}$.
\end{proof}

\section{Lyapunov function and orbital stability}\label{sec3}

In this section, we will prove our main theorem. 
Before presenting the result itself, we need to introduce some notation and give some preliminary tools. In fact, since equation $(\ref{rDE})$ is invariant under translations, we define the orbit generated by $\phi$ as
\begin{equation}\label{orbit}
\Omega_{\phi}=\{\phi(\cdot+r);\ r\in\mathbb{R}\}.
\end{equation}
In $H_{per}^{\frac{m}{2}}([0,L])$, we introduce the pseudometric $d$ by
$$d(f,g)=\inf\{||f-g(\cdot+r)||_{\frac{m}{2}},r\in\mathbb{R}\}.$$
It is to be observed that, by definition, the distance between $f$ and $g$ is measured by the distance between $f$ and the orbit generated by $g$. 
Given $\varepsilon>0$, the $\varepsilon$-neighborhood of $\Omega_{\phi}$ is defined by
$$\Omega_{\phi}^{\varepsilon}=\{v\in H_{per}^{\frac{m}{2}}([0,L]);\ d(v,\Omega_{\phi})<\varepsilon\}.$$

The precise definition of orbital stability is given next.

\begin{definition}\label{stadef}
Let $\phi$ be a traveling wave solution for \eqref{rDE}. We say that $\phi$ is orbitally stable in $H_{per}^{\frac{m}{2}}([0,L])$ provided that, given $\ve>0$, there exists $\delta>0$ with the following property: if $u_0\in H_{per}^{s}([0,L])$, for some $s\geq \frac{m}{2}$, satisfies $\|u_0-\phi\|_{\frac{m}{2}}<\delta$, then the solution, $u(t)$, of \eqref{rDE} with initial condition $u_0$ exist for all $t\geq0$ and satisfies
$$
d(u(t),\Omega_\phi)<\ve, \qquad \mbox{for all}\,\, t\geq0.
$$
Otherwise, we say that $\phi$ is orbitally unstable in $H_{per}^{\frac{m}{2}}([0,L])$.
\end{definition}

\begin{remark}
Note that in Definition \ref{stadef} we are implicitly assuming that a global well-posedness result for \eqref{rDE} holds in some Sobolev space $H_{per}^{s}([0,L])$, for some $s\geq \frac{m}{2}$.
\end{remark}

\begin{lemma}\label{lema4}
Given $\rho>0$ and $v\in \Omega_{\phi}^{\rho}$, there exists $r_1\in \mathbb{R}$ such that 
\begin{equation}\label{est11}
||v-\phi(\cdot+r_1)||_{\frac{m}{2}}<\rho
\end{equation}
and
\begin{equation}\label{est12}
(v-\phi(\cdot+r_1),\phi'(\cdot+r_1))_{\frac{m}{2}}=0.
\end{equation}
\begin{proof}
Let us define the function $f(r)=||v-\phi(\cdot+r)||_{\frac{m}{2}}^2$, $r\in\mathbb{R}$. Since $v$ and $\phi$ are periodic, $f$ assumes its minimum at a point $r_1$, which without lost of generality can be assumed to belong to the interval $[0,L)$. Thus, the smoothness of $f$  guarantees the existence of   $r_1\in[0,L)$  such that $(\ref{est11})$ and $(\ref{est12})$ hold.
\end{proof}
\end{lemma}

As we already said, the proof of our main result is based on the construction of a Lyapunov function. Let us make clear what we mean by this in our context.

\begin{definition}\label{lydef}
A function $V:H_{per}^{\frac{m}{2}}([0,L])\to\mathbb{R}$ is said to be a Lyapunov function for the orbit $\Omega_\phi$ if the following properties hold.
\begin{itemize}
\item[(i)] There exists $\rho>0$ such that $V:\Omega_\phi^\rho\to\mathbb{R}$ is of class $C^2$ and, for all $v\in\Omega_\phi$,
$$
V(v)=0 \quad\mbox{and}\quad \qquad V'(v)=0.
$$
\item[(ii)] There exists $c>0$ such that, for all $v\in\Omega_\phi^\rho$,
$$
V(v)\geq c[d(v,\Omega_\phi)]^2.
$$
\item[(iii)] For all $v\in\Omega_\phi^\rho$, there hold
$$
\langle V'(v),\partial_xv\rangle=0.
$$
\item[(iv)] If $u(t)$ is a global solution of the Cauchy problem associated with \eqref{rDE} with initial datum $u_0$, then $V(u(t))=V(u_0)$, for all $t\geq0$.
\end{itemize}
\end{definition}

The next step in then the construction of a Lyapunov function. To do so, let us set
$$
q_1=G(\phi), \qquad q_2=Q(\phi).
$$
Given any positive constant $\sigma$, define $V:H_{per}^{\frac{m}{2}}([0,L])\to\R$ by
\begin{equation}\label{l12}
V(v)=G(v)-q_1+\sigma(Q(v)-q_2)^2.
\end{equation}

We now prove the main result of this subsection.

\begin{proposition}\label{lyalemma2}
Assume that the Cauchy problem associated with $(\ref{rDE})$ is globally well-posed in a convenient Sobolev space $H_{per}^s([0,L])$, $s\geq \frac{m}{2}$. There exists $\sigma>0$ such that the functional defined in \eqref{l12} is a Lyapunov function for the orbit $\Omega_\phi$.
\end{proposition}
\begin{proof}
Since $G$ and $Q$ are smooth conserved quantities of \eqref{rDE} and the Cauchy problem associated with \eqref{rDE} is assumed to be globally well-posed, it is clear that part (iv) in Definition \ref{lydef} is satisfied and $V$ is of class $C^2$. Since $V(\phi)=0$ and the functionals $G$ and $Q$ are invariant by translations, we have $V(v)=0$, for all $v\in\Omega_\phi$. In addition, because
\begin{equation}\label{l13}
\langle V'(u),v\rangle=\langle G'(u),v\rangle+2\sigma(Q(u)-q_2)\langle Q'(u),v\rangle
\end{equation}
for all $u,v\in H_{per}^{\frac{m}{2}}([0,L])$, and $\phi$ is a critical point of $G$, it also clear that $V'(\phi)=0$. By observing that
$\phi(\cdot+r)$ is also a critical point of $G$, it then
follows that $V'(v)=0$ for all $v\in \Omega_\phi$. Part (i) of
Definition \ref{lydef} is also established for any $\rho>0$.

Since
$$
Q(v(\cdot+r))=Q(v), \qquad G(v(\cdot+r))=G(v)
$$
for all $r\in\R$ and $v\in H_{per}^{\frac{m}{2}}([0,L])$, we can take the derivatives with respect to $r$ in order to see that part (iii) in Definition \ref{lydef} is also satisfied for any $\rho>0$.

Finally, let us check part (ii). From \eqref{l13}, we obtain
$$
\langle V''(u)v,v\rangle=\langle G''(u)v,v\rangle+2\sigma(Q(u)-q_2)\langle Q''(u)v,v\rangle+2\sigma\langle Q'(u),v\rangle^2.
$$
From $(\ref{dual3})$ and the fact that $Q'(\phi)\in C_{per}^{\infty}([0,L])$, enable us to conclude
\begin{equation}\label{V}
\begin{array}{llll}
\langle V''(\phi)v,v\rangle&=&\langle G''(\phi)v,v\rangle+2\sigma\langle Q'(\phi),v\rangle^2\\\\
&=&(\RR^{-1}G''(\phi)v,v)_{\frac{m}{2}}+2\sigma\big(Q'(\phi),v)^2\\\\
&=&(Sv,v)_{\frac{m}{2}}+2\sigma\la\mathcal{I}Q'(\phi),v\ra^2\\\\
&=&(Sv,v)_{\frac{m}{2}}+2\sigma(\mathcal{R}^{-1}\mathcal{I}Q'(\phi),v)_{\frac{m}{2}}^2.
\end{array}
\end{equation}
 Hence,
$$
\langle
V''(\Phi)v,v\rangle=\big(Sv,v)_{\frac{m}{2}}+2\sigma\big(\mathcal{R}^{-1}\I Q'(\phi),v)_{\frac{m}{2}}^2.
$$
Thus, from Lemma \ref{lema3} we deduce the existence of  positive constants $c_6$ and $\sigma$ such that
\begin{equation}\label{l14}
\langle V''(\Phi)v,v\rangle\geq c_6\|v\|_{\frac{m}{2}}^2,
\end{equation}
for all $v\in\{\phi'\}^\perp$. Since $V$ is of class $C^2$, a Taylor expansion gives
$$
V(v)=V(\phi)+\langle V'(\phi),v-\phi\rangle+\frac{1}{2}\langle V''(\phi)(v-\phi),v-\phi\rangle+h(v),
$$
where $h$ is a function satisfying
$$
\lim_{v\to\phi}\frac{h(v)}{\|v-\phi\|^2_{\frac{m}{2}}}=0.
$$
Thus, we can select $\rho>0$ such that
\begin{equation}\label{l15}
|h(v)|\leq\frac{c_6}{4}\|v-\phi\|^2_{\frac{m}{2}},  \quad \mbox{for all}\, v\in B_\rho(\phi).
\end{equation}
By noting that $V(\phi)=0$ and $V'(\phi)=0$, and using \eqref{l14} and \eqref{l15}, it follows that
\begin{equation}\label{l16}
\begin{split}
V(v)&=\frac{1}{2}\langle V''(\phi)(v-\phi),v-\phi\rangle+h(v)\\
&\geq \frac{c_6}{2}\|v-\phi\|^2_{\frac{m}{2}}-\frac{c_6}{4}\|v-\phi\|^2_{\frac{m}{2}}\\
&=\frac{c_6}{4}\|v-\phi\|^2_{\frac{m}{2}}\\
&\geq \frac{c_6}{4}[d(v,\Omega_\phi)]^2,
\end{split}
\end{equation}
provided that $\|v-\phi\|_{\frac{m}{2}}<\rho$ and $v-\phi\in\{\phi'\}^\perp$.

Now take any $v\in\Omega_\phi^\rho$. Since $\rho>0$, from Lemma
\ref{lema4} there exist $r_1\in\R$ such that
$u:=v(\cdot-r_1)\in B_\rho(\phi)$ and
$$
\big(v-\phi(\cdot+r_1),\phi'(\cdot+r_1)\big)_{\frac{m}{2}}=0,
$$
which mean that
$\|u-\phi\|_{\frac{m}{2}}<\rho$ and $u-\phi\in\{\phi'\}^\perp$. Consequently, \eqref{l16} implies
$$
V(v)=V(u)\geq \frac{c_6}{4}[d(u,\Omega_\phi)]^2=\frac{c_6}{4}[d(v,\Omega_\phi)]^2.
$$
This proves part (ii) and completes the proof of the proposition.
\end{proof}

Now we prove our orbital stability result.

\begin{theorem}\label{stateo}
Under assumption ($H_0$)-($H_2$), the periodic traveling wave solution $\phi$ of \eqref{rDE} is orbitally stable in $H_{per}^{\frac{m}{2}}([0,L])$.
\end{theorem}
\begin{proof}
	Having disposed a Lyapunov function, the proof of orbital stability is quite standard (see e.g., \cite{st} and \cite{np}). For the sake of completeness we give the main steps. Fix $\ve>0$ and let $V:\Omega_\phi^\rho\to\R$ be the Lyapunov function given in Proposition \ref{lyalemma2}.  By using the continuity of $V$ and the fact that $V(\phi)=0$, we obtain the existence of $\delta\in(0,\rho)$ such that
	$$
	V(v)=V(v)-V(\phi)<c\min\left\{\frac{\rho^2}{4},\ve^2 \right\}, \qquad v\in B_\delta(\phi),
	$$
	where $c>0$ is the constant appearing in Definition \ref{lydef}.
	Since  $V$ is invariant by translations, 
	\begin{equation}\label{e1}
	V(v)<c\min\left\{\frac{\rho^2}{4},\ve^2 \right\}, \qquad v\in \Omega_\phi^\delta.
	\end{equation}
	Let $u_0\in H_{per} ^{\frac{m}{2}}([0,L])$ be a function such that $u_0\in B_\delta(\phi)$. Since it is assumed a convenient global well-posedness result, the solution, say $u(t)$, of the Cauchy problem associated to \eqref{rDE} with initial data $u_0$ is defined for all $t\geq0$. Let $J$ be the interval defined as
	$$
	J=\{s>0; \, u(t)\in \Omega_\phi^\rho \,\,\mbox{for\,\,all}\,\,t\in[0,s)\}.
	$$
	The continuity of $u(t)$ immediately implies that $J\neq\emptyset$ and $\inf J=0$. Let us show that $J=[0,\infty)$, that is, $s^*:=\sup J=\infty$. Assume by contradiction that $s^*<\infty$. Parts (ii) and (iv) of Definition \ref{lydef} give
	$$
	c[d(u(t),\Omega_\phi)]^2\leq V(u(t))=V(u_0)<c\frac{\rho^2}{4},
	$$
	for all $t\in[0,s^*)$, where in the last inequality we have used the fact that $u_0\in B_\delta(\phi)$ and \eqref{e1}. Thus, we deduce that $d(u(t),\Omega_\phi)<\rho/2$ for all $t\in[0,s^*)$. It is clear that the continuity of $u(t)$ implies the continuity of the function $t\mapsto d(u(t),\Omega_\phi)$. Consequently, $d(u(s^*),\phi)\leq\rho/2$. The continuity of $u(t)$ implies again that $\sup J>s^*$, which is a contradiction. Therefore, $J=[0,\infty)$ and
	$$
	c[d(u(t),\Omega_\phi)]^2\leq V(u(t))=V(u_0)<c\ve^2
	$$
	for all  $t\geq0$. The proof of the theorem is thus completed.
\end{proof}

\begin{remark}
The theory presented here can be extended to study the orbital stability of solitary-wave solutions for \eqref{rDE}. In particular, Theorem \ref{stateo} extends mutatis mutandis when a parametrization, depending on the wave speed, of the solitary waves is not available. For similar results in this direction see \cite{al} (see also \cite{np}).
\end{remark}

\section{Sufficient conditions for  orbital stability}\label{sec4}

In this section, we present sufficient conditions to obtain the key assumption in $(H_2)$, by assuming that $(H_0)$ and $(H_1)$ hold.

\begin{proposition}\label{prop2}
Assume that there is $\Phi\in H_{per}^{m}([0,L])$ such that $\langle\mathcal{L}\Phi,\varphi\rangle=0$, for all $\varphi\in \Upsilon_0=\{u\in H_{per}^{m}([0,L]);\ (Q'(\phi),u)=0\}$, and
\be\label{defiI}
(\mathcal{L}\Phi,\Phi)<0.
\ee
Then, there is a constant $c_7>0$ such that
$$(\mathcal{L}v,v)\geq c_7||v||^2,$$
for all $v\in \Upsilon_0$ such that $( v,\phi')=0$.   
\end{proposition}
\begin{proof}
We shall give only a sketch of the proof. From assumption $(H_0)$ one has
\be\label{decomp}L_{per}^2([0,L])=[\chi]\oplus [\phi']\oplus P,\ee
where $\chi$ satisfies $||\chi||=1$ and $\mathcal{L}\chi=-\lambda_0^2\chi$, $\lambda_0\neq0$.  By using the arguments  in \cite[page 278]{kato}, we obtain that $$( \mathcal{L}p,p)\geq c_8||p||^2,\ \ \ \ \ \mbox{for all}\ p\in H_{per}^m([0,L])\cap P,$$
where $c_8$ is a positive constant.

In view of $(\ref{decomp})$, we write
$$\
\Phi=a_0\chi+b_0\phi'+p_0,\ \ \ \ \ a_0,b_0\in\mathbb{R},
$$
where $p_0\in H_{per}^m([0,L])\cap P$. Now, since $\phi'\in \ker (\mathcal{L})$, $\mathcal{L}\chi=-\lambda_0^2\chi$, and $(\mathcal{L}\Phi,\Phi)<0$, we obtain
\be\label{4.3}
(\mathcal{L} p_0,p_0)=(\mathcal{L}(\Phi-a_0\chi-b_0\phi'),\Phi-a_0\chi-b_0\phi')
=(\mathcal{L}\Phi,\Phi)+a_0^2\lambda_0^2<a_0^2\lambda_0^2.
\ee

Taking $\varphi\in \Upsilon_0$ such that $||\varphi||=1$ and $( \varphi,\phi')=0$, we can write $\varphi=a_1\chi+p_1$, where $p_1\in H_{per}^m([0,L])\cap P$. Thus,
\begin{equation}\label{4.4}
0=(\mathcal{L}\Phi,\varphi)=( -a_0\lambda_0^2\chi +\mathcal{L}p_0,a_1\chi+p_1)
=-a_0a_1\lambda_0^2+(\mathcal{L}p_0,p_1).
\end{equation}
From \eqref{4.3} and \eqref{4.4} it is not difficult to check that $(\mathcal{L}\varphi,\varphi)>0$. All details and the rest of the proof can be found in \cite[Lemma 5.1]{bss}  or in \cite[Lemma 7.8]{an}.
\end{proof}

By combining Theorem \ref{stateo} with Proposition \ref{prop2}, one sees that in order to obtain the orbital stability, under assumptions ($H_0$)-($H_1$), it suffices to obtain an element $\Phi\in H_{per}^{m}([0,L])$ satisfying the conditions in Proposition \ref{prop2}. As an immediate application we have the following.

\begin{corollary}\label{coro12}
	Assume that the Cauchy problem associated to $(\ref{rDE})$ with $f(v)=\frac{v^2}{2}$ is globally well-posed in a convenient Sobolev space $H_{per}^s([0,L])$, $s\geq \frac{m}{2}$. Assume also that $\mathcal{M}$ satisfies \eqref{A1A2} with $\kappa_0=0$.  If hypotheses $(H_0)$ and $(H_1)$ hold, then the periodic wave $\phi$ is orbitally stable provided that $M(\phi)>\omega L$.
\end{corollary}
\begin{proof}
	This result is proved by taking $(\mu,\nu)=(\omega,-1)$ in $(\ref{modquant})$ and $\Phi=1$ in Proposition $\ref{prop2}$.
\end{proof}

So far, the obtained results do not depend on any parametrization of the solutions of \eqref{ode-wave}. In particular, the results apply for any periodic solution. When the solutions of \eqref{ode-wave} can be parametrized by the parameters $\omega$ and $A$, the element $\Phi$ can be found, as we will se below, by analyzing a suitable quadratic form. So, in what follows we make the following assumption.\\

\begin{enumerate}
\item[\textit{$(H_3)$}] Suppose that there is an open subset $\mathcal{O}\subset\mathbb{R}^2$ such that $(\omega,A)\in\mathcal{O}\mapsto \phi_{(\omega,A)}\in C_{per}^{\infty}([0,L])$ is a smooth surface of periodic traveling waves with fixed period $L>0$ which solve $(\ref{ode-wave})$. Moreover, we also assume that the spectral assumption in $(H_0)$ remains valid for $\phi:=\phi_{(\omega,A)}$, $(\omega,A)\in \mathcal{O}$.
\end{enumerate}
\bigskip

 Next, having hypothesis $(H_3)$ in mind, we define
$$
\eta:=\frac{\partial}{\partial\omega}\phi_{(\omega,A)},\ \qquad\beta:=\frac{\partial}{\partial A}\phi_{(\omega,A)},
$$
and set
 $$
 M_{\omega}(\phi)=\int_0^{L}\eta dx,\qquad  M_{A}(\phi)=\int_0^{L}\beta dx,
 $$
 and
 $$
 F_{\omega}(\phi)=\frac{1}{2}\int_0^{L}\frac{\partial}{\partial \omega}(\phi_{(\omega,A)}^2)dx, \qquad F_{A}(\phi)=\frac{1}{2}\int_0^{L}\frac{\partial}{\partial{A}}(\phi_{(\omega,A)}^2)dx.
 $$

 We have a simple connection among $\mathcal{L}$ and $M_{\omega}(\phi)$, $F_{\omega}(\phi)$ and $M_A(\phi)$. In fact, differentiating $(\ref{ode-wave})$ with respect to $\omega$ and $A$, we obtain respectively, $\mathcal{L}\eta=-\phi$ and $\mathcal{L}\beta=-1$. Since $1,\phi\in[\phi']^{\bot}$, and $\mathcal{L}:[\phi']^{\bot}\rightarrow[\phi']^{\bot}$ is invertible, we have
\begin{equation}\label{relLMA}
M_{\omega}(\phi)=-(\mathcal{L}^{-1}\phi,1),\ \  M_{A}(\phi)=-(\mathcal{L}^{-1}1,1)\ \  \mbox{and}\ F_{\omega}(\phi)=-(\mathcal{L}^{-1}\phi,\phi).
\end{equation}

 \indent Next result gives us a sufficient condition to obtain $(\ref{defiI})$.

\begin{proposition}\label{propKpos}
Let $\Delta:\R^2\to\R$ be the function defined as
$$
\Delta(x,y)=x^2M_A(\phi)+xy(M_\omega(\phi)+F_A(\phi))+y^2F_\omega(\phi).
$$
Assume that there is $(a,b)\in\R^2$ such that $\Delta(a,b)>0$. Then there is $\Phi\in H_{per}^{m}([0,L])$ such that   $(\mathcal{L}\Phi,\varphi)=0$, for all $\varphi\in \Upsilon_0$, and
$$
( \mathcal{L}\Phi,\Phi)<0.
$$
\end{proposition}
\begin{proof}
It suffices to define  $\Phi:=a\beta+b\eta$. Indeed, since $\mathcal{L}\beta=-1$ and $\mathcal{L}\eta=-\phi$, it is clear that  $(\mathcal{L}\Phi,\varphi)=0$, for all $\varphi\in \Upsilon_0$, and
\[
\begin{split}
 (\mathcal{L}\Phi,\Phi)&=(-a-b\phi,a\beta+b\eta)\\
 &=-(a^2M_A(\phi)+abM_\omega(\phi)+abF_A(\phi)+b^2F_\omega(\phi))\\
 &=-\Delta(a,b).
\end{split}
\]
The proof is thus completed.
\end{proof}

Combining assumptions $(H_1)$ and $(H_3)$ with the result in Propositions $\ref{prop2}$ and $\ref{propKpos}$, we are able to establish the following stability result.

\begin{theorem}\label{teostab}
Assume that the Cauchy problem associated with $(\ref{rDE})$ is globally well-posed in a convenient Sobolev space $H_{per}^s([0,L])$, $s\geq \frac{m}{2}$. If the assumptions $(H_1)$ and $(H_3)$ are valid, the periodic wave $\phi$ is orbitally stable provided that there is $(a,b)\in\mathbb{R}^2$ such that $\Delta(a,b)>0$. In particular, the stability result occurs if at least one of the following statements hold:
\begin{itemize}
	\item[(i)]  $M_A(\phi)=-(\mathcal{L}^{-1}1,1)>0$,
	\item[(ii)] $F_{\omega}(\phi)=-(\mathcal{L}^{-1}\phi,\phi)>0,$
	\item[(iii)] $M_{\omega}(\phi)^2-F_{\omega}(\phi)M_A(\phi)=(\mathcal{L}^{-1}\phi,1)^2-(\mathcal{L}^{-1}\phi,\phi)(\mathcal{L}^{-1}1,1)>0$.
	\end{itemize}
\end{theorem}
\begin{proof}
The first part of the theorem is clear. Parts (i) and (ii) are obtained by considering in Proposition $\ref{propKpos}$, $(a,b)=(1,0)$ and $(a,b)=(0,1)$, respectively. To obtain (iii), we need to derive equation $(\ref{ode-wave})$ with respect to $\omega$, multiply the result by $\phi$ and then integrate the final result over $[0,L]$. With these arguments in hand, we have from the self-adjointness of $\mathcal{M}$ and $(\ref{ode-wave})$,
\begin{equation}\label{omegawave}
-AM_{\omega}(\phi)+\int_0^L f(\phi)\eta dx-\int_0^Lf'(\phi)\phi\eta dx+\int_0^L\phi^2 dx=0.
\end{equation}
Similarly, if we derive  $(\ref{ode-wave})$ with respect to $A$, we conclude that
\begin{equation}\label{Awave}
-AM_{A}(\phi)+\int_0^L f(\phi)\beta dx-\int_0^Lf'(\phi)\phi\beta dx+\int_0^L\phi dx=0.
\end{equation}
Thus, deriving $(\ref{omegawave})$ with respect to $A$, $(\ref{Awave})$ with respect to $\omega$, and  comparing the obtained results, we get
\begin{equation}\label{FAMomega}
F_A(\phi)=M_{\omega}(\phi).
\end{equation}
Hence,
$$
\Delta(x,y)=x^2M_A(\phi)+2xyM_\omega(\phi)+y^2F_\omega(\phi)=(x,y)\,S\,(x,y)^T.
$$
where $S$ is the symmetric matrix
$$
S:=\left[\begin{array}{llll} F_{\omega}(\phi)\ \ M_{\omega}(\phi)\\
M_{\omega}(\phi)\ \ M_{A}(\phi)\end{array}\right].
$$ 
Since $\det(S)=-\big(M_{\omega}(\phi)^2-F_{\omega}(\phi)M_A(\phi)\big)<0$ it follows that $S$ has two real eigenvalues with opposite sign, which implies that the quadratic form $\Delta$ is indefinite. Consequently, there is $(a,b)$ such that $\Delta(a,b)>0$.

\end{proof}

\begin{remark}\label{det}
The case $M_{\omega}(\phi)^2-F_{\omega}(\phi)M_{A}(\phi)< 0$ deserves to be highlighted. In such a situation, both  $M_{A}(\phi)$ and $F_{\omega}(\phi)$ are nonzero and have the same sign. For one hand, if both are positive, we have the orbital stability from Theorem $\ref{teostab}$. On the other hand, if  both are negative, we can use the arguments in \cite[Theorem 1]{DK} to conclude that 
\begin{equation}
\label{Kham}K_{\rm{Ham}}=n(\mathcal{L})-n((\mathcal{L}^{-1}1,1))-n(\mathcal{D})
\end{equation}
where $\mathcal{D}$ is given in \eqref{Dmatrix}, $n(\mathcal{L})$ indicates the number of negative eigenvalues (counting multiplicities) of $\mathcal{L}$, and
$$
n(s)=\left\{\begin{array}{llll}
1,\ \ \ \mbox{\rm{if}}\ \ s>0,\\
0,\ \ \ \mbox{\rm{if}}\ \ s<0.
\end{array}\right.
$$
From \eqref{relLMA} we have
$$
\mathcal{D}=\frac{1}{M_A(\phi)}\left(M_{\omega}(\phi)^2-F_{\omega}(\phi)M_{A}(\phi)\right)>0.
$$
Thus  $K_{\rm{Ham}}=n(\mathcal{L})-n(-M_A(\phi))-n(\mathcal{D})=1$. Due to the Hamiltonian eigenvalue symmetry, that is, if $\lambda$ is an eigenvalue so are $-\lambda$ and $\pm\overline{\lambda}$, it must be the case that $k_c$ and $k_i^{-}$ are even numbers. As a consequence, $K_{\rm{Ham}}=k_r=1$, which means that the periodic wave $\phi$  is spectrally unstable.
\end{remark}

\begin{remark}
Recall that in \cite{BJK1}, \cite{BJK2}, and \cite{DK} the authors obtained their stability results under the assumption that $(\mathcal{L}^{-1}1,1)\neq0$. It should be pointed out that, in Theorem \ref{teostab} we can obtain the orbital stability without such an assumption. %%(see our applications below for a specific example). 
Thus, Theorem \ref{teostab} can be seen as an improvement of those works.
\end{remark}

In many practical situations, the parameters $\omega$ and $A$ in \eqref{ode-wave} are not independent. Instead, both are dependent of a third parameter, say, $\xi$ with $\xi$ belonging to some open interval. So, in this situation, instead of having a smooth surface as in ($H_3$), we have a smooth curve of periodic waves. Our conditions are still sufficient to obtain the orbital stability.

\begin{corollary}\label{coro2}
Assume that the Cauchy problem associated with $(\ref{rDE})$ is globally well-posed in a convenient Sobolev space $H_{per}^s([0,L])$, $s\geq \frac{m}{2}$. Suppose that $(H_1)$ and $(H_3)$ are valid with the difference that $\omega$ and $A$ depend smoothly on $\xi$. If $\Phi=\frac{\partial}{\partial \xi}\phi_{(\omega(\xi),A(\xi))}$, the periodic wave $\phi=\phi_{(\omega(\xi),A(\xi))}$ is orbitally stable in $H_{per}^{\frac{m}{2}}([0,L])$ provided that 
\begin{equation}\label{solcriterio}
(\mathcal{L}\Phi,\Phi)=-\frac{d A}{d \xi} \frac{d}{d\xi}M(\phi)-\frac{d \omega}{d \xi}\frac{d}{d\xi}F(\phi)<0.
\end{equation}
\end{corollary}
\begin{proof}
Taking $\mu=\frac{\partial A}{\partial \xi}$ and $\nu=\frac{\partial \omega}{\partial \xi}$ in $(\ref{modquant})$, the result follows by applying Proposition $\ref{prop2}$.
\end{proof}

\begin{remark}
Note that \eqref{solcriterio} is a generalization of the well known criterion for the orbital stability of solitary-wave solutions for equations of the form \eqref{rDE} (under suitable spectral conditions as in ($H_0$)). Indeed, for solitary waves, it is clear that $A=0$ and $\omega=\xi$. Thus, \eqref{solcriterio} immediately reduces to 
$$
\dfrac{d}{d\omega}\int \phi^2dx>0.
$$
The interested reader will find all details, for instance, in \cite{bss}.
\end{remark}

\subsection{Applications} In order to illustrate our results, we will present several applications taking into account different scenarios.

\subsubsection{The KdV equation}
Let us start our applications with a very simple example. By assuming $\mathcal{M}=-\partial_x^2$ and $f(v)=\frac{v^2}{2}$, \eqref{rDE} reduce to the well known Korteweg-de Vries equation,
\begin{equation}\label{kdv1}
u_t+uu_x+u_{xxx}=0.
\end{equation}
An explicit family of periodic traveling waves of \eqref{kdv1} is well known. For instance, in \cite{ABS}, the authors presented periodic waves with the zero mean property given by
\begin{equation}\label{cnoidalwaves}
\phi(x)=\beta\left({\rm{dn}}^2\left(\frac{2K(k)}{L}x,k\right)-\frac{E(k)}{K(k)}\right),
\end{equation}
where $L>0$ is fixed and $\beta$ depends smoothly on the wave speed $\omega>0$.  Here, dn stands for the dnoidal elliptic function,  $K$ and $E$ indicate the complete elliptic integrals of first and second kind, respectively, and both of them depend on the elliptic modulus $k\in(0,1)$. By recalling that ${\rm dn^2}=1-k^2{\rm cn^2}$, where cn is the cnoidal elliptic function, these solutions are indeed the cnoidal solutions studied by Benjamin in \cite{benjamin}.

In this case, the constant $A$ is given by $A=\frac{1}{2L}\int_0^L\phi(x)^2dx$ and so, it depends smoothly on $\omega$ as well. Assumption $(H_1)$ is easily obtained by using the arguments in Remark $\ref{garding}$. Here, the linearized operator reads as $\mathcal{L}=-\partial_x^2+\omega-\phi$. To obtain the spectral properties as in hypothesis $(H_0)$ it is necessary to use the classical Floquet theory combined with the spectral theory associated with the Lam\'e equation (see \cite[Section 5]{ABS} for details). According to Corollary $\ref{coro2}$ let us consider $\xi=\omega$ and $\Phi=\frac{\partial}{\partial \omega}\phi$. The arguments in \cite[Theorem 5.2]{ABS} established that $\frac{\partial }{\partial \omega}\int_0^L\phi(x)^2dx>0$, for all $\omega>0$. Therefore, since $\phi$ has zero mean, from \eqref{solcriterio}, one has 
$$
(\mathcal{L}\Phi,\Phi)=-\frac{d }{d \omega}\int_0^L\phi^2dx<0,\ \ \ \ \omega>0.
$$
 Thus, the periodic cnoidal wave in $(\ref{cnoidalwaves})$ is orbitally stable in $H_{per}^1([0,L])$.

\subsubsection{The generalized KdV equation}

If $\mathcal{M}=-\partial_x^2$ and $f(v)=\frac{v^{p+1}}{p+1}$, the generalized Korteweg-de Vries equation,
 \begin{equation}\label{mkdv1}
 u_t+u^{p}u_x+u_{xxx}=0,
 \end{equation}
 emerges. The periodic wave solutions must satisfy
\begin{equation}\label{ode-wavekdv}
-\phi''+\omega\phi-\frac{1}{p+1}\phi^{p+1}+A=0.
\end{equation}
By multiplying \eqref{ode-wavekdv} by $\phi$ and integrating once, it can be written in the quadrature form
\begin{equation}\label{quadrakdv}
-\phi'^2+\omega\phi^2-\frac{2}{(p+1)(p+2)}\phi^{p+2}+2A\phi+2B=0,
\end{equation}
with $B$ appearing as another constant of integration. Thus, the periodic solutions of \eqref{quadrakdv} can be smoothly parametrized by the triple $(A,B,\omega)$, that is, $\phi=\phi(\cdot; A,B,\omega)$ (see the details in \cite{johnson}).

In this specific case, the period $L$ is a real function which also depends on the parameters $(A,B,\omega)$ and therefore, we can not directly use the results contained in Proposition $\ref{propKpos}$ and Theorem $\ref{teostab}$. However, recall that the arguments in Section 3 do not depend on any parametrization of the solutions of \eqref{ode-wavekdv}. Thus, if we assume that assumption $(H_0)$ holds\footnote{In \cite{johnson}, it is proved that $L_{B}>0$ is sufficient to obtain the spectral properties.} (from Remark $\ref{garding}$, we see that assumption $(H_1)$ is easily verified, at least for $\omega>0$), we only need to prove the existence of an element $\Phi$  satisfying the conditions in Proposition \ref{defiI}.

 In \cite[page 1935]{johnson}, the author has defined the  periodic function $\Phi$ by
\begin{equation}\label{Phimat}
\Phi=\left|\begin{array}{llllllllll}
\phi_A & & L_A & & M_A(\phi)\\
\phi_B & & L_B & & M_B(\phi)\\
\phi_{\omega} & & L_{\omega} & & M_{\omega}(\phi)
\end{array}\right|,
\end{equation}
where $\phi_A=\frac{\partial}{\partial A}\phi$, $\phi_B=\frac{\partial}{\partial B}\phi$, etc. In addition, a straightforward calculation gives  that $\mathcal{L}\Phi$ can be expressed in terms of convenient Jacobian determinants and the periodic wave $\phi$ as 
\begin{equation}\label{JacobL}
\mathcal{L}\Phi=-\{L,M\}_{B,\omega}-\{L,M\}_{A,B}\phi,
\end{equation}
where
$$
\{L,M\}_{B,\omega}=\left|\begin{array}{cc}
 L_B &  M_B(\phi)\\
 L_\omega &  M_\omega(\phi)
\end{array}\right|
\qquad \mbox{and} \qquad
\{L,M\}_{A,B}=\left|\begin{array}{cc}
L_A &  M_A(\phi)\\
L_B &  M_B(\phi)
\end{array}\right|.
$$
Thus, by taking $\mu=\{L,M\}_{B,\omega}$ and $\nu=\{L,M\}_{A,B}$ in \eqref{modquant}, it is easy to see that $(\mathcal{L}\Phi,\varphi)=0$, for all $\varphi\in \Upsilon_0$. In addition, 
\begin{equation}\label{LPhiPhi}
(\mathcal{L}\Phi,\Phi)=-\{L,M\}_{A,B}\{L,M,F\}_{A,B,\omega}.
\end{equation}
As a consequence of Proposition \ref{prop2} and Theorem \ref{stateo}, $\phi$ is orbitally stable in $H^1_{per}([0,L])$
provided that $\{L,M\}_{A,B}\{L,M,F\}_{A,B,\omega}>0$. Hence, under assumption ($H_0$) we recover the orbital stability results in \cite[Lemma 4.1]{johnson}. It is to be noted that, according to Remark 9 in \cite{johnson}, only the positivity of the product on the right-hand side of \eqref{LPhiPhi} is needed. The sign of each determinant does not play any role for the orbital stability.\\

\subsubsection{The Intermediate Long Wave (ILW) equation.}
  Now, we present a simple way to prove the orbital stability of periodic waves with the mean zero property for the Intermediate
Long Wave equation,
\begin{equation}\label{ILW}
u_t+2uu_x+\delta^{-1}u_x-(\mathcal{T_\delta}u)_{xx}=0,\ \ \ \ \ \ \delta>0.
\end{equation}
The linear operator
$\mathcal{T_\delta}$ is defined by
$$
\mathcal{T_\delta}u(x)=\frac{1}{L} \text{p.v.} \int_{-L/2}^{L/2}
\Gamma_{\delta, L}(x-y) u(y)dy,
$$
where p.v. stands for  the Cauchy principal value of the integral and
$$
\Gamma_{\delta, L}(\xi)=-i\sum_{n\neq 0} \coth\left(\frac{2\pi
n\delta}{L}\right)e^{2in \pi \xi/L}.
$$

It is well known that if parameter $\delta$ goes to infinity, equation in $(\ref{ILW})$ converges, at least formally, to the Benjamin-Ono equation
 \begin{equation}\label{BO}
u_t+2uu_x-\mathcal{H}u_{xx}=0,
\end{equation}
with  $\mathcal{H}$ denoting the periodic Hilbert transform and  defined for $L$-periodic functions  as
\begin{equation}\label{symbBO}
\mathcal{H}f(x)=\frac{1}{L} \text{p.v.} \int_{-L/2}^{L/2}
{\rm cot}\Big[\frac{\pi(x-y)}{L}\Big]f(y)dy,
\end{equation}
 On the other hand, if $\delta$ goes to zero, then equation in $(\ref{ILW})$ converges formally to the corresponding KdV equation.

\indent The orbital stability of periodic waves related to the Benjamin-Ono equation was determined in \cite{natali} (but we can apply the method presented in this paper to give a simpler proof).

\indent By looking for periodic waves of the form $u(x,t)=\phi(x-\omega t)$, $\omega>0$, we see that $\phi$ must satisfy the non-local equation
\begin{equation}\label{travkdv}
\omega\phi-\phi^2+\mathcal{M}\phi+A=0, \end{equation}
 where $A$  is the integration constant given by $A=\frac{1}{L}\int_0^L\phi(x)^2dx$ and $\mathcal{M}=\mathcal{T_\delta}\partial_x-\frac{1}{\delta}$. The symbol of $\mathcal{M}$ is given by
 $$
 \theta(\kappa)=\frac{2\pi \kappa}{L}\coth\left( \frac{2\pi \kappa \delta}{L}\right)-\frac{1}{\delta}, \qquad \kappa\in\mathbb{Z}.
 $$
 Note that by using the relation (see \cite[Lemma 4.1]{absa})
 \begin{equation}\label{abdel}
-\frac{1}{\delta}+2\pi|y|\leq 2\pi\coth(2\pi\delta y)\leq \frac{1}{\delta}+2\pi|y|, \quad \delta>0, y\in\R,
 \end{equation}
we can take $\upsilon_2=2\pi/L$ in \eqref{A1A2}. Also, choose $\kappa_0$ sufficiently large so that $\delta>L/(\kappa_0\pi)$. This implies that 
$$
\vartheta_0:=\frac{2\pi}{L}-\frac{2}{\kappa_0 \delta}>0.
$$
By choosing any $\upsilon_1>0$ such that $\upsilon_1<\vartheta_0$, \eqref{abdel} implies that, for $|\kappa|\geq \kappa_0$,
$$
\upsilon_1|\kappa|\leq -\frac{2}{\delta}+\frac{2\pi|\kappa|}{\delta}\leq \frac{2\pi \kappa}{L}\coth\left( \frac{2\pi \kappa \delta}{L}\right)-\frac{1}{\delta}=\theta(\kappa).
$$
Thus \eqref{A1A2} holds with $m=1$ and the natural space to study \eqref{ILW} is then $H^{\frac{1}{2}}_{per}([0,L])$.

 The smooth $L$-periodic  solution with zero mean is given explicitly by
 \begin{equation}\label{solilw}
\phi(x)
=\displaystyle\frac{2K(k)i}{L}\displaystyle\left[Z\displaystyle
\left(\displaystyle\frac{2K(k)}{L}(x-i\delta);k\right)-Z\displaystyle
\left(\displaystyle\frac{2K(k)}{L}(x+i\delta);k\right)\right],\ \
\end{equation}
where $Z$ is the Jacobi Zeta function. In \cite{natali1} it was shown that the linearized operator $\mathcal{L}=\mathcal{M}+\omega-2\phi$ fulfills the spectral property required in assumption $(H_0)$. In addition, for a fixed $\delta>0$, it was numerically determined  that $\frac{\partial }{\partial\omega}\int_0^L\phi(x)^2dx>0$, for all $\omega>0$. Now, if one sets $\xi=\omega$ and $\Phi=\frac{\partial}{\partial \omega}\phi$ in Corollary $\ref{coro2}$, we obtain that 
$$(\mathcal{L}\Phi,\Phi)=\left(-\frac{\partial}{\partial\omega}A-\phi,\frac{\partial }{\partial\omega}\phi\right)=-\frac{1}{2}\frac{\partial}{\partial\omega}\int_0^L\phi(x)^2dx<0.$$
Thus, the periodic wave $\phi$ is orbitally stable in $H^{\frac{1}{2}}_{per}([0,L])$.

\section{Extensions to regularized equations}\label{sec5}

\indent The arguments  in the previous sections can also be used to determine sufficient conditions for the orbital stability of periodic waves related to the regularized equation
\begin{equation}\label{regDE}
u_t+u_x+(f(u))_x+(\mathcal{M}u)_t=0,
\end{equation}
where $f$ and $\mathcal{M}$ are as before. Also in this case, periodic traveling waves of $(\ref{regDE})$ are solutions of the form $u(x,t)=\phi(x-\omega t)$. Again, if we substitute this kind of special solution into $(\ref{regDE})$ one has, after integration,
\begin{equation}\label{travregDE}
(\omega-1)\phi-f(\phi)+\omega \mathcal{M}\phi+A=0.
\end{equation}

\indent Equation $(\ref{regDE})$ has at least three conserved quantities,
\begin{equation*}
P(u)=\frac{1}{2}\int_{0}^{L}\big(u\mathcal{M}u-W(u)\big)dx,
\end{equation*}
\begin{equation*}
F(u)=\frac{1}{2}\int_{0}^{L}\big(u\mathcal{M}u+u^2\big)dx,
\end{equation*}
and 
\begin{equation*}
M(u)=\int_{0}^{L}udx.
\end{equation*}
Here, function $W$ denotes the primitive of $f$, that is, $W'=f$.\\
\indent Define $G(u)=P(u)+(\omega-1)F(u)+AM(u)$ and $Q(u)=\mu M(u)+\nu F(u)$, $\mu,\nu\in\mathbb{R}$. By assuming that $(H_0)$, $(H_1)$ and $(H_2)$ hold, we can repeat the arguments in Sections 2 and 3 in order to obtain similar results about the positivity of the linearized operator $\mathcal{L}=\omega\mathcal{M}+(\omega-1)-f'(\phi)$, as well as the construction of the Lyapunov function and, consequently, the orbital stability.

In addition,  let us assume that a similar condition as in $(H_3)$ is satisfied. If  $\eta$ and $\beta$ are as above,  consider
$$
 M_{\omega}(\phi)=\int_0^{L}\frac{\partial}{\partial\omega}\phi_{(\omega,A)}dx,\qquad  M_{A}(\phi)=\int_0^{L}\frac{\partial}{\partial A}\phi_{(\omega,A)}dx,
 $$
 and
 $$
 F_{\omega}(\phi)=\frac{1}{2}\int_0^{L}\frac{\partial}{\partial{\omega}}(\phi_{(\omega,A)}\mathcal{M}\phi_{(\omega,A)}+\phi_{(\omega,A)}^2)dx,\ \ F_{A}(\phi)=\frac{1}{2}\int_0^{L}\frac{\partial}{\partial{A}}(\phi_{(\omega,A)}\mathcal{M}\phi_{(\omega,A)}+\phi_{(\omega,A)}^2)dx.
 $$
Thus, the results of Proposition $\ref{propKpos}$ and Theorem $\ref{teostab}$ can be obtained. Moreover, if we assume that $\phi_{(\omega,A)}$  is a periodic traveling wave solution which solves equation $(\ref{travregDE})$ with $\omega$ and $A$ depending on the parameter $\xi$, a similar result as in Corollary $\ref{coro2}$ may be established.

\end{document}